\documentclass{amsart}

\usepackage[english]{babel}
\usepackage{graphicx}

\usepackage{amsmath, latexsym, mathtools, amsthm, amsfonts, amssymb}
\usepackage{array, hyperref, url, tcolorbox}
\usepackage{fullpage}

\hypersetup{
     colorlinks=true,
     linkcolor=blue,
     filecolor=blue,
     citecolor=red,      
     urlcolor=cyan,
     }

\newtheorem{theorem}{Theorem}[section]
\newtheorem{lemma}[theorem]{Lemma}
\newtheorem{proposition}[theorem]{Proposition}
\newtheorem{corollary}[theorem]{Corollary}

\theoremstyle{definition}
\newtheorem{definition}[theorem]{Definition}
\newtheorem{example}[theorem]{Example}
\newtheorem{fact}[theorem]{Fact}

\theoremstyle{remark}
\newtheorem{remark}[theorem]{Remark}
\newtheorem{notation}[theorem]{Notation}
\newtheorem{note}[theorem]{Note}

\numberwithin{equation}{section}

\newcommand{\prend}{$\hfill \Box$}

\newcommand{\vertiii}[1]{{\left\vert\kern-0.25ex\left\vert\kern-0.25ex\left\vert #1 
    \right\vert\kern-0.25ex\right\vert\kern-0.25ex\right\vert}}

\begin{document}


\title[Probabilistic Pad\'e]
{
On the clustering of Pad\'e zeros and poles of random power series
}

\author{Stamatis Dostoglou}
\address{Department of Mathematics, University of Missouri, Columbia, MO, 65211.}
\email{dostoglous@missouri.edu}

\author{Petros Valettas}
\address{Departments of Mathematics and Electrical Engineering \& Computer Science, University of Missouri, Columbia, MO, 65211.}
\email{valettasp@missouri.edu}
\thanks{P.V. is supported in part by the Simons Foundation (grant \#638224).}






\begin{abstract}
	We estimate non-asymptotically the probability of uniform clustering around the unit circle of the zeros of the $[m,n]$-Pad\'e approximant	of a random power series $f(z) = \sum_{j=0}^\infty a_j z^j$ for $a_j$ independent, with finite first moment, 
	and L\'evy function satisfying ${\mathcal L}(a_j , \varepsilon) \leq K\varepsilon$. Under the same assumptions we show that  almost surely $f$  has infinitely 
	many zeros in the unit disc, with the unit circle serving as a natural boundary for $f$. For $R_m$ the radius of the largest disc containing at most $m$ zeros of $f$, a deterministic result of Edrei implies that in our setting the poles of the $[m,n]$-Pad\'e approximant almost surely cluster uniformly at the circle of radius $R_m$ as $n \to \infty$ and $m$ stays fixed, 
	and we provide almost sure rates of converge of these $R_m$'s to $1$.  We also show that our results on the clustering of the zeros hold for 
	log-concave vectors $(a_j)$ with not necessarily independent coordinates.
\end{abstract}


\maketitle



\section{Introduction}

The $[m,n]$-Pad\'e approximant of a power series $f(z) = \sum_{j=0}^\infty a_j z^j$ is the rational function
 \begin{equation}
       \mathfrak{P}_{mn}(z) 
       = 
       \frac{P_{mn}(z)}{Q_{mn}(z)}
       = 
       \frac{p_0 + p_1 z + \ldots + p_mz^m}{q_0 + q_1 z + \ldots + q_n z^n}
\end{equation}
with Taylor series at $0$ that matches $f$ as much as possible, as introduced by Frobenius in \cite{Frob}.
The aim of this article is to examine the behavior of Pad\'e approximants when the coefficients $a_j$, 
and therefore the $p$'s and $q$'s, are random variables.  Our interest is in the behavior, with high probability, 
of the random zeros and poles of $\mathfrak{P}_{mn}$ and how these cluster around specific geometric loci, cf.\ \cite{Froi}, \cite{Gil97}, \cite{Gil99}. 
Under our assumptions here on the randomness of the $a_j$'s, and for the range of $m$ and $n$ we examine, this locus turns out to be the unit circle.  

In particular, as the $[m,n]$-Pad\'e approximant of the 
power series $\sum_{j =0}^\infty a_jz^j$ is the same as the Pad\'e approximant of the polynomial $\sum_{j=0}^N a_j z^j$ for $N \geq m+n$, see section \ref{prelims},  we show as Theorem \ref{thm:adv}  that: Whenever $(a_j)_{j=0}^N$ for $N \geq m+n$ is an isotropic random vector 
with independent coordinates and with a certain anti-concentration property, see \eqref{themainanticoncentration}, then the zeros of the 
numerator $P_{mn}$ of the $[m,n]$-Pad\'e approximant of the random polynomial $\sum_{j =0}^N a_jz^j$ cluster uniformly around the 
unit circle for $m$ and $n$ in a certain range, see \eqref{therange}.  
We understand uniform clustering  as in Szeg\"o \cite{S}, Rosenbloom \cite{R}, and Erd\H os-Tur\'an, \cite{ET}. We measure how much 
clustering is achieved either via the Erd\H os-Tur\'an ratio \eqref{The Ratio} or via the distance of the empirical measure of the zeros of 
the $P_{mn}$ from the uniform measure on the unit circle in the bounded Lipschitz metric \eqref{the metric}. The result is non-asymptotic 
in that we find the range of $m$ and $n$ for which the desired degree of clustering happens. 
Theorem \ref{thm:adv} follows from the more general Theorem  \ref{thm:main-1}, whose proof relies on the 
deterministic Erd\H os-Tur\'an result \cite{ET}, a standard application of Jensen's formula on zeros of holomorphic functions \cite{J}, 
the calculation of the coefficients of the Pad\'e numerator using Toeplitz matrices \eqref{eq:Toep-1}, \eqref{eq:Toep-2}, 
and our estimate on the probability of  invertibility of Toeplitz matrices, Proposition \ref{prop:sb-det}. 

Theorem \ref{radius} provides another way of understanding why the zeros of Pad\'e numerators cluster uniformly around the unit circle. 
One of the conclusions of that Theorem is that the unit circle is almost surely a natural boundary for the power series $f(z) = \sum_{j=0}^\infty a_j z^j$ 
when the $a_j$'s are independent, have moment bounds, and anti-concentrate. 
Our proof of this relies on the Ryll-Nardzewski theorem \cite[p.\ 41]{K} and complements the classical result on symmetric random $a_j$'s \cite[p.\ 39]{K} 
and the results of Breuer and Simon \cite{BS} on bounded $a_j$'s.
Now a deterministic result of Edrei \cite{E}, recasting Hadamard \cite{Had1892} and generalizing Szeg\"o \cite{S}, 
shows that for meromorphic $f$, as $n$ stays constant and $m\to \infty$, the zeros of the Pad\'e numerator $P_{mn}$ cluster, 
up to subsequence, around the boundary of the largest disc that contains not more than $n$ poles of $f$. In the presence 
of a natural boundary all such discs are the unit disc, therefore clustering takes place around the unit circle. Note that Edrei's 
deterministic result is asymptotic and holds for some $m$-subsequence whereas our probabilistic Theorem \ref{thm:adv} holds for all $m$'s and $n$'s that satisfy our non-asymptotic condition \eqref{therange}. Note also that part of Edrei's proof is devoted to showing that the Toeplitz determinants that calculate the Pad\'e numerators and denominators are not singular for his subsequences. In our work, this is reflected by our estimate on the probability of the invertibility of random Toeplitz matrices, Proposition \ref{prop:sb-det}.

Regarding the poles of the Pad\'e approximant $\mathfrak{P}_{mn}$, 
we rely in Section \ref{as} on the fact the the $[m,n]$-denominator for $f$ is the $[n,m]$-numerator of $1/f$ 
and appeal directly to Edrei's result for $1/f$. For this, we need information on the position of the poles of $1/f$, 
equivalently of the zeros of $f$. This we  garner in Theorem \ref{roots} that shows that almost surely $f$ has 
infinitely many zeros in the unit disc and estimates how fast the radius $R_m$  of the largest disc that contains at most $m$ zeros converges to $1$. 
The assumptions here are the same as for Theorem \ref{radius}. The proof uses Jensen's formula \cite{J}. 
The application of Edrei's result for almost any realization of the random series takes place in Theorem \ref{polecluster} 
showing that for fixed $m$ the zeros of $Q_{mn}$ cluster uniformly  as $n \to \infty$ around the circle of radius $R_m$, with $1 - R_m = O(\log m / m)$.

For $n=0$ the Pad\'e approximants are, of course, nothing but the Taylor polynomials that approximate the power series. For random coefficients, the study of these fits in the extensive literature on random polynomials, see for example \cite{HN, IZ, SS}.

Some background on Pad\'e approximants and preliminary probabilistic results are included in Section \ref{PTET}. 
The final Section \ref{LogConcave} provides the estimate on the probability of the invertibility of random Toeplitz matrices 
when the $a_j$'s are not independent but come from a log-concave random vector. An Appendix shows how the 
Erd\H os-Tur\'an ratio controls the bounded Lipschitz distance of the empirical measure of the zeros of a polynomial from 
the uniform measure on the unit circle.  

Whereas we made some effort to give precise values to the various universal constants that appear in the estimates 
we do not claim that they are the best possible. When we do not specify the constants we reserve the right to change them 
freely using the same notation.

\section{Pad\'e, Toeplitz, and Erd\H os -Tur\'an} \label{PTET}

Pad\'e approximants are inextricably linked with Toeplitz matrices, 
and a result of Erd\H os  and Tur\'an is a standard way to show clustering of roots of polynomials. 
In this section we present these connections along with preliminary results, some probabilistic, that we use later. 

\subsection{Pad\'e approximants and formulas} \label{prelims}

Recall that given a power series
\begin{equation}
        f(z)= \sum_{j=0}^\infty a_j z^j,
\end{equation}
the $[m,n]$-Pad\'e approximant of $f$ is the (unique) rational function
\begin{equation}
        \frak{P}_{mn}=P_{mn}/Q_{mn}
\end{equation}
with $P_{mn}$ and $Q_{mn}$ polynomials on $\mathbb{C}$ of degree at most
$m$ and $n$, respectively, and $Q_{mn}(z) \not\equiv 0$, such that 
\begin{align} \label{eq:F-def}
		f(z) Q_{mn}(z) - P_{mn}(z) = \sum_{j \geq m+n+1} c_{j}z^{j}, 
	\end{align}
for some $c_j$'s, \cite{Frob}, \cite{Gra}.
(For the definition in \cite{Ba} and its place in the probabilistic approach, see Remark \ref{on det} below.)

To determine the coefficients ${\bf p}= (p_0, \ldots, p_m)$ of $P_{mn}$ and ${\bf q}= (q_0, \ldots, q_n)$ of $Q_{mn}$, 
the Frobenius formulation \eqref{eq:F-def} leads to the following linear systems:	
	\begin{align} \label{eq:Toep-1}		
		\begin{bmatrix} 
			a_0&0&0&\ldots&0 \\
			a_1 & a_0 &0&\ldots&0\\
			\vdots \\
			a_m & a_{m-1} & \ldots &\dots& a_{m-n} 
		\end{bmatrix} {\bf q} \equiv C_m^{(n)} {\bf q}= {\bf p}, 	
	\end{align}
and
	\begin{align} \label{eq:Toep-2}
		\begin{bmatrix} 
			a_{m+1} & a_m & a_{m-1} & \ldots & a_{m-(n-1)} \\
			a_{m+2} & a_{m+1} & a_m & \ldots & a_{m-(n-2)} \\
			\vdots & \vdots & \vdots & \ddots & \vdots \\
			a_{m+n} & a_{m+(n-1)} & a_{m+(n-2)} & \ldots & a_m
		\end{bmatrix} {\bf q} \equiv T_m^{(n)} {\bf q} = {\bf 0}.
	\end{align}
Note that $C_m^{(n)}$ is an $(m+1)\times (n+1)$ matrix, whereas $T_m^{(n)}$ is an $n \times (n+1)$ matrix. It is a standard convention that $a_l =0$ when $l$ is negative.

Prominent in what follows will be the $n\times n$ square matrices $A_m^{(n)}$ (following the notation in \cite{E}), a submatrix of $T_m^{(n)}$, 
\begin{equation} \label{assocmatr}
      A_m^{(n)}
      =
     \begin{bmatrix} 
			 a_m & a_{m-1} & \ldots & a_{m-(n-1)} \\
			 a_{m+1} & a_m & \ldots & a_{m-(n-2)} \\
			\vdots & \vdots & \ddots  & \vdots \\
		 a_{m+(n-1)} & a_{m+(n-2)} & \ldots & a_m
		\end{bmatrix}.
\end{equation}
These are Toeplitz matrices, i.e.\ matrices with constant entries on each diagonal. 

\begin{remark} \label{on det}
It is important to consider the case 
\begin{equation}  \label{nonzerodeterminant}
     \det A_m^{(n)} \neq 0, \quad \text{for all\ } m, n.
\end{equation}
Then the Pad\'e polynomials can be written as 
\begin{equation} \label{PQformulas}
\begin{split}
       P_{mn}(z ) 
       &= 
       a_0 + \ldots + \frac{\det A_m^{(n+1)} }{\det A_m^{(n)}}z^m,\\
       Q_{mn}(z) 
       & = 
       1 - \frac{\det (T_m^{(n)}[2])}{{\det A_m^{(n)}}} 
       +
       \ldots 
       + 
       (-1)^n \frac{\det A_{m+1}^{(n)} }{\det A_m^{(n)}}z^n,
\end{split}
\end{equation}
with $T_m^{(n)}[k]$ the square matrix that results from $T_m^{(n)}$ 
from \eqref{eq:Toep-2} after the $k$-th column is omitted, see \cite[Corollary 1, p.\ 18]{Gra}. 

Note that \eqref{nonzerodeterminant} then implies that the top order coefficients of \eqref{PQformulas} 
do not vanish. This in turn implies that any pair of polynomials satisfying \eqref{eq:F-def} are constant multiples of those in \eqref{PQformulas}. 
In this case the Baker condition $Q_{mn}(0) \neq 0$  \cite[p.\ 6]{Ba} is satisfied and we also have
\begin{align} \label{Baker def}
		f(z) - \frac{P_{mn}(z)}{Q_{mn}(z)} = \sum_{j \geq m+n+1} \hat c_j z^j.
	\end{align}
As in our power series $\sum_{j\geq 0} a_j z^j$ will have random coefficients, 
we examine next conditions for \eqref{nonzerodeterminant} to be satisfied almost always, or at least with high probability. 
In particular, pathological examples like $f(z) = 1 + \sum_{j \geq 2} z^j$ when $P_{11}(z) = z$ and $Q_{11}(z) =z$ satisfy \eqref{eq:F-def}
but not \eqref{Baker def}, either almost surely or with high probability will not happen.  
\end{remark}

\subsection{Random Toeplitz determinants} We now quantify 
the singularity of random Toeplitz matrices. Recall first the {\it L\'evy concentration} function for a random variable $\xi$ for $\delta>0$:
	\begin{align}
		{\mathcal L}(\xi, \delta) =\sup_{t\in \mathbb R} \mathbb P(|\xi-t|<\delta).
	\end{align}
	
This measure of dispersion was introduced in the works of Doeblin \& L\'evy \cite{DL}, Kolmogorov \cite{Kolmo}, and Esseen \cite{Ess} 
concerning the spread of sums of independent random variables. In recent years it has become an indispensable
tool in the study of quantitative invertibility of (unstructured) random matrices, see \cite{RV-icm} for a detailed exposition.
In the present work it occupies a central role, too, since it quantifies the invertibility of random patterned matrices (Toeplitz), and it is related
to the existence of a naturally boundary for a random power series.

\begin{lemma} \label{small polynomial values} For any (deterministic) monic polynomial $P$ with 
$\deg P=d$, for any random variable $\xi$, and for any $\varepsilon >0$, we have
	\begin{align}
		\mathbb P \left ( |P(\xi)| < \varepsilon^d \right) \leq d\, \mathcal{L}(\xi, \varepsilon).
	\end{align}
\end{lemma}

\begin{proof} 
Write $P(z)= \prod_{j=1}^d (z-r_j)$, for $r_j$ the (complex) roots of $P$. Then
	\begin{equation}
		\begin{split}
      			\mathbb P \left ( |P(\xi)| < \varepsilon^d \right ) \leq \sum_{j=1}^d \mathbb P (|\xi - r_j| < \varepsilon) \leq 
			\sum_{j=1}^d \mathbb P \left( \left| \xi - {\rm Re}(r_j) \right| < \varepsilon \right)	\leq 		
			d \mathcal{L}(\xi, \varepsilon),
		\end{split} 
	\end{equation}	using the union bound.
\end{proof}

\begin{proposition} [quantitative invertibility] \label{prop:sb-det}
For ${\bf a}=(a_k)_{0\leq k \leq 2n -2}$ a random vector with independent coordinates, and for $A$ the Toeplitz matrix
\begin{equation}
      A = \begin{bmatrix} 
			 a_{n-1} & a_{n-2} & \ldots & a_{0} \\
			 a_{n} & a_{n-1} & \ldots & a_{1} \\
			\vdots & \vdots & \ddots  & \vdots \\
		 a_{2n -2} & a_{2n -3} & \ldots & a_{n-1}
		\end{bmatrix},
\end{equation}
and for any $\varepsilon>0$, 
		\begin{align}
			\mathbb P \left( |\det A|^{1/n} < \varepsilon \right) \leq n {\mathcal L}(a_{n-1}, \varepsilon).
		\end{align}
\end{proposition}

\begin {proof} Write 
\begin{equation}
        A = a_{n-1} I + \sum_{k \neq  n-1} a_k B_k
\end{equation}
where each $B_k$ has $1$'s on a single, not the main, diagonal and $0$'s everywhere else. 
The random matrices $a_{n-1} I$ and $B:=\sum_{k \neq n-1}a_k B_k$ are independent. Conditioning on $a_k$ for all $k \neq n-1$,
\begin{equation}
\begin{split}
        \mathbb P \left( |\det A |^{1/n} <\varepsilon \right) 
		&= 
		\mathbb E 
		\left[ 
		\mathbb P\left( | \det(a_{n-1}I + B) | < \varepsilon^n \mid (a_k)_{ k \neq n-1} \right)  
		\right]\\
		&= 
		\mathbb E 
		\left[ 
		\mathbb P \left( |P_{-B} (a_{n-1})| <\varepsilon^n \mid (a_k)_{ k \neq n-1} \right)
		\right]
\end{split}
\end{equation}
for $P_{-B}(\cdot)$ the characteristic polynomial of a matrix $-B$. 
As $P_{-B}$ is monic, apply Lemma \ref{small polynomial values} conditionally to conclude. 
\end{proof}

\subsection{Clustering of zeros} 

Let $P(z) = \alpha_0 + \alpha_1z+ \ldots + \alpha_Nz^N$, $\alpha_N\neq 0$ be a polynomial with roots $z_1, \ldots, z_N$, and let 
	\begin{align}
		\nu_P= \frac{1}{N}\sum_{j=1}^N \delta_{z_j},
	\end{align}
 the normalized zero-counting measure associated with $P$.

For a sequence of polynomials $P_m$, each of degree $m$, we write $\nu_m$ for $\nu_{P_m}$.

The following makes precise the meaning of clustering of zeros around a circle, following Szeg\"o  \cite{S}, 
Rosenbloom \cite{R},  Erd\H os-Tur\'an \cite{ET}, and Edrei \cite{E}.

\begin{definition}[clustering] \label{Clustering defintion}
The zeros of a sequence of polynomials $P_m$, each of degree $m$, cluster uniformly around the circle of radius $r$ if both of the following are satisfied:

\smallskip

\noindent (a) {\it Radial clustering}. For any $\rho>0$, for the annulus 
\begin{align}
R(r,\rho) 
= 
\left\{ 
z\in \mathbb C :  
(1-\rho)r<|z| < (1+\rho)r 
\right\},
\end{align}
we have
\begin{equation}
  \nu_m(R(r,\rho)) \to 1,\quad m\to \infty.
\end{equation}
\noindent (b) {\it Angular clustering}.  For $0 \leq \theta < \phi < 2\pi$, for the angular sector sector
\begin{align}
S(\theta, \phi)= \left\{ z\in \mathbb C : \theta < {\rm Arg}\ z \leq \phi \right\},
\end{align}
we have
\begin{equation}
     \nu_m(S(\theta, \phi)) \to \frac{\phi - \theta}{2\pi}, \quad m\to \infty. 
\end{equation}
\end{definition}

Our goal is to establish such clustering for the polynomial numerators and denominators of Pad\'e approximants of random Taylor series,
and to estimate how much clustering takes place at large but finite degree.

As is well known, both radial and angular clustering of polynomials around the \textit{unit circle} are controlled by comparing the end coefficients to all of the coefficients:
For any polynomial $P(z) = \alpha_0 + \alpha_1z + \ldots + \alpha_Nz^N$ with $\alpha_0 \alpha_N\neq 0$, set
 	\begin{align} \label{The Ratio}
		L(P) : = \frac{ |\alpha_0| + |\alpha_1| + \ldots + |\alpha_N|}{\sqrt{ |\alpha_0 | \cdot |\alpha_N|}}. 
	\end{align} 
Then for radial clustering around the unit circle we have:
\begin{proposition} \label{prop:J2}
For $P(z) = \alpha_0 +\alpha_1z + \ldots + \alpha_Nz^N$ with $\alpha_0\alpha_N\neq 0$ and for $0<\rho \leq 1$,
		\begin{align}
			1- \nu_P(R(1,\rho)) \leq \frac{\log L(P)}{\rho N} . 
		\end{align}
\end{proposition}
A proof of this is part of Theorem 1 in \cite[p.\ 268]{R}. Edrei reproves it \cite[p.\ 264]{E}. See also \cite{HN} 
for a more recent take connected to probabilistic considerations independent of Pad\'e approximants. 
All these use the classical formula of Jensen \cite{J} which, for a polynomial $P$ links the measure $\nu_P$ and the average of
$P$ on a circle. For a useful presentation of Jensen's formula, including the case of zeros on the circle, see \cite[\S 4.8]{AN}. 
We state the formula here in a form that will be useful later.

\begin{proposition}[Jensen Formula] \label{JensenForm}
For $f$ holomorphic on the closed disc of radius $r$ and with $f(0) \neq 0$,
 \begin{align} \label{eq:J}
 	\frac{1}{2\pi} \int_0^{2\pi} \log| f(re^{i\theta})| \, d\theta 
	+
	\sum_{j=1}^n \log\frac{|z_j|}{r}
	= 
	\log|f(0)| ,
 \end{align}
 where $z_1, \ldots, z_n$ are the zeros of $f$ in the interior of the disc of radius $r$, repeated according to multiplicity.
\end{proposition}

As for angular clustering, we have the celebrated Erd\H os-Tur\'an result \cite{ET}, see also \cite{R}. For a more recent treatment see \cite{AB}.

\begin{proposition} [Erd\H os-Tur\'an] \label{prop:ET}
	For $P(z)=\alpha_0 + \alpha_1z+\ldots+\alpha_Nz^N$ with $\alpha_0\alpha_N\neq 0$ and for $0 \leq \theta < \phi < 2\pi$,
		\begin{align}
			\left|  \frac{\phi - \theta}{2\pi} - \nu_P(S(\theta, \phi)) \right| \leq 16 \sqrt{ \frac{\log L(P)}{N}} .
		\end{align}
\end{proposition}
		
For $\mu$ the uniform measure on the unit circle $\mathbb{T} = \{ z \in \mathbb{C}: |z|=1\}$, Proposition \ref{prop:J2} and Proposition \ref{prop:ET} show that 
$\nu_P$ is close to $\mu$ in certain sense. This proximity can be measured in terms of the bounded Lipschitz metric 
 	\begin{align} \label{the metric}
		d_{\rm BL} (\mu, \nu_P)  
		= 
		\sup \left\{ \left| \int f \, d\mu - \int f \, d\nu \right| : f\in B_1 \right \}, 
	\end{align} 
	for $B_1:=\{f: \mathbb{C} \to \mathbb R \mid \|f\|_{\rm BL}:= \|f\|_{\rm Lip} + \|f\|_\infty \leq 1\}$. It is well known that this metric induces 
	weak convergence in the space of probability distributions, c.f. \cite[\S 8.3]{Bog}.
We then have 
 
 \begin{proposition} \label{prop:disc}
 	For $P(z)= \alpha_0 + \alpha_1z +\ldots + \alpha_Nz^N$ with $\alpha_0\alpha_N\neq 0$, 
		\begin{align}
			d_{\rm BL}(\nu_P, \mu) \leq 32 \left( \frac{\log L(P)}{N} \right)^{1/4}.
		\end{align}
\end{proposition}

For completeness, we include a proof of this proposition in the Appendix \ref{app:disc}.

\subsection{Erd\H os-Tur\'an for Pad\'e numerators and denominators}
We now estimate the Erd\H os-Tur\'an ratio \eqref{The Ratio} in the case of Pad\'e numerators of 
$f(z)=\sum_{j=0}^\infty a_jz^j$ in terms of the coefficients $a_j$. 
Toeplitz matrices inevitably appear and we assume here that their determinants do not vanish. 
As mentioned in Remark \ref{on det}, this assumption will hold in what follows either almost always or with high probability.

From \eqref{PQformulas} the coefficient vectors of the numerator and denominator 
of the $[m,n]$-Pad\'e  of $f$ will be ${\bf p}=(a_0, p_1, \ldots, p_m)$, ${\bf q}=(1,q_1,\ldots,q_n)$ . From \eqref{eq:Toep-1} we have 
	\begin{align} \label{eq:2-1}
		\|{\bf p}\|_1 
		= 
		\|C_m^{(n)} {\bf q} \|_1 
		\leq 
		\|{\bf q} \|_1 \cdot \max_{j\leq n} \|C_m^{(n)} e_j\|_1 
		= 
		\| {\bf q} \|_1 \cdot \sum_{j=0}^m |a_j|.
	\end{align}
Along with the expression for $q_n$ from  \eqref{PQformulas}, this estimates the Erd\H os-Tur\'an ratio of $P_{mn}$ in terms of the ratio of $Q_{mn}$:
	\begin{equation}  \label{PratioQratio}
	\begin{split}
		L(P_{mn}) 
		&\leq 
		\frac{\| {\bf q} \|_1}{\sqrt{|a_0| |q_n|}} 
		\cdot
		 \left( \frac{| \det A^{(n)}_{m+1} | } { |\det A^{(n+1)}_{m}| } \right)^{1/2} 
		 \cdot \sum_{j=0}^m |a_j| \\
		&= 
		\frac{\sum_{j=0}^m |a_j|}{\sqrt{|a_0|}} \cdot 
		 \left( \frac{| \det A^{(n)}_{m+1} | } { |\det A^{(n+1)}_{m}| } \right)^{1/2}
		\cdot L(Q_{mn}).
	\end{split}
	\end{equation}	
On the other hand, the expression for $q_k$ from \eqref{PQformulas} gives 
	\begin{align}
		L(Q_{mn}) 
		= \|{\bf q}\|_1 
		\left( \frac{|\det A_m^{(n)} |}
		{|\det A_{m+1}^{(n)}|} \right)^{1/2} &
		=
		\left(1 + \sum_{k=1}^n 
		\frac{\det(T_m^{(n)}[k+1])}{|\det A_m^{(n)}|} \right)
		\cdot
		 \left( \frac{|\det A_m^{(n)} |}
		{|\det A_{m+1}^{(n)}|} \right)^{1/2}. 
	\end{align}
Then \eqref{PratioQratio} becomes 
	\begin{align}
	L(P_{mn}) &
	\leq \frac{\sum_{j=0}^m |a_j|}{\sqrt{|a_0|}} 
	\cdot 
	\left(1 + \sum_{k=1}^n \frac{\det(T_m^{(n)}[k+1])}{|\det A_n^{(m)}|} \right)
	\cdot 
	\left( \frac{|\det A_m^{(n)} |}{|\det A_{m}^{(n+1)}|} \right)^{1/2} \\
        & \leq  
        \frac{\sum_{j=0}^m |a_j|}{\sqrt{|a_0|}} 
        \frac
        {\sqrt {n \det\left(T_m^{(n)}T_m^{(n)^\ast}\right)}}
        {\left( |\det A_m^{(n)} | \cdot |\det A_{m}^{(n+1)}| \right)^{1/2}},
	\end{align} 
after using the Cauchy-Schwarz inequality to go to the sum of the squares of the sub-determinants and the 
Cauchy-Binet formula for  $\det\left(T_m^{(n)}T_m^{(n)^\ast}\right)$:
\begin{equation}
       \det\left(T_m^{(n)}T_m^{(n)^\ast}\right)
       =
       \sum_{k=0}^n \det\left(T_m^{(n)}[k+1]\right)^2.
\end{equation}
Finally, the arithmetic--geometric mean inequality (applied on the singular values of $T_m^{(n)}$) gives
	\begin{align} \label{eq:2-7}
		L(P_{mn}) 
		&\leq 
		\frac{\sum_{j=0}^m |a_j|}{\sqrt{|a_0|}} 
		\cdot 
		\frac
		{n^{1/2} \left(n^{-1/2} \vertiii{T_m^{(n)} }_1\right)^n }
		{\left( |\det A_m^{(n)} | \cdot |\det A_{m}^{(n+1)}| \right)^{1/2}},	
	\end{align} 
where $\vertiii{T_m^{(n)} }_1 := \sum_{j=1}^{n+1} \|T_m^{(n)}(e_j) \|_1$.

\section{Non-asymptotic clustering of roots} \label{non-as}

Given a random power series of the form
\begin{equation}
        f(z)= \sum_{j=0}^\infty a_j z^j, 
        \quad 
        a_j:(\Omega, \mathbb{P}) \to \mathbb{C}
        \ \text{random variables},
\end{equation}
it is clear from \eqref{eq:Toep-1} and \eqref{eq:Toep-2} that calculating the $[m,n]$-Pad\'e approximant of 
$f$ is the same as caclulating the $[m,n]$-Pad\'e approximant of 
$f_N = \sum_{j=0}^N a_j z^j$ for any $N\geq m +n$.  As our focus in this section is on non-asymptotic results, 
we work here with $f_N$. We deal with full power series in the next section. 
The following is our main probabilistic, non-asymptotic estimate for the Erd\H os-Tur\'an ratio of the numerator 
of the Pad\'e approximants when the $a_j$'s are independent random variables. 
It is applied when $ {\mathcal L}(a_j, \varepsilon)\leq K \varepsilon$ in Theorem \ref{thm:adv} and we show its use for discrete distributions in Example \ref{example}. 

\begin{theorem} [Bound on the Erd\H os-Tur\'an ratio] \label{thm:main-1}	
Let  $m,n\in \mathbb N$, $1 \leq \gamma<\infty $, and $\varepsilon$ and $b$ in $(0,1)$.

Then for any $N\geq m+n$ and $(a_j)_{j=0}^N$ random vector  with independent coordinates, $\mathbb E|a_j| \leq \gamma$  and ${\mathcal L}(a_j ,\varepsilon)\leq b$ for all $j$, the numerator $P_{mn}$ of the $[m,n]$-Pad\'e approximant of $f_N(z) = \sum_{k=0}^N a_jz^j$ satisfies 
	\begin{align} \label{eq:control-ET}
		\mathbb P \left(  \left\{ \log L(P_{mn}) > \log m + 2n\log \left(\frac{4\gamma}{b\varepsilon}\right) \right\}  \right) \leq 5nb.
	\end{align} 
\end{theorem}

\begin{proof}
Take $5nb <1$, as otherwise there is nothing to prove. 	
	In view of the estimate \eqref{eq:2-7}, consider the following events:
	\begin{align}
		{\mathcal E}_0:=&\{|a_0| < \varepsilon \}, \\
		{\mathcal E}_1:=& \left\{ \sum_{j=0}^m |a_j| > t(m+1)\gamma \right\}, \quad 	 {\mathcal E}_2:= \left\{ \vertiii{T_{m}^{(n)} }_1 > tn(n+1) \gamma \right\} \\
		{\mathcal E}_3:=& \{  |\det A_m^{(n)} | \cdot |\det A_{m}^{(n+1)}| < \varepsilon^{2n}\},
	\end{align} 
where $t>0$ will be suitably chosen below. 
Note that 
		\[
			P({\mathcal E}_0) \leq b, \quad \mathbb P({\mathcal E}_3) \leq n{\mathcal L}(a_m,\varepsilon) + (n+1){\mathcal L}(a_m,\varepsilon)   \leq 3nb,
		\] 
by the assumption on $\mathcal{L}$ and Proposition \ref{prop:sb-det}, and that  
		\[
			\mathbb P({\mathcal E}_1\cup {\mathcal E}_2) \leq \frac{2}{t},
		\] by Markov's inequality. Therefore in the complement of $\bigcup_{i=0}^3 {\mathcal E}_i$ we have 
	\[
	 	L(P_{mn}) \leq \frac{2t m \gamma n^{1/2} }{\varepsilon} \cdot \left(\frac{2t \gamma n^{3/2} }{\varepsilon} \right)^n, 
	\] 
with probability greater than $1-b-3nb -2/t$.
For $t\geq \dfrac{2}{nb}$ this gives
\begin{equation}
      L(P_{mn}) \leq  m\left( \frac{2n t \gamma}{\varepsilon}\right)^{2n},
\end{equation}	
	with probability greater than $1-5nb$.
In particular, the choice $t= 2/(nb)$ yields 
	\begin{align}
		L(P_{mn}) \leq m \left( \frac{4 \gamma }{b\varepsilon}\right)^{2n},
	\end{align} with probability greater than $1-5nb$, as claimed. 
\end{proof}

The following theorem holds with the same integrability conditions as in Theorem \ref{thm:main-1}. We present it here for isotropic random vectors, a standard class of random vectors in high-dimensional probability. Recall that a random vector is called {\it isotropic} if it is centered and its covariance matrix is the identity, see e.g., \cite[Definition 2.3.7]{BGVV}. 
Applying Theorem \ref{thm:main-1} to this case we have
\begin{theorem} \label{thm:adv}
	Let $m,n\in \mathbb N$, $K\geq 1$ and $\delta\in (0,1)$. If $m,n$ satisfy
		\begin{equation}\label{therange}			
			m\geq C\delta^{-4}\, n\, \log (eKn/ \delta),
		\end{equation} then for any $N\geq m+n$, and for any 
	isotropic random vector ${\bf a}= (a_j)_{j=0}^N \in \mathbb R^{N+1}$ with independent coordinates, and
	\begin{equation} \label{themainanticoncentration}
		 {\mathcal L}(a_j, \varepsilon)\leq K \varepsilon,\quad \text{for all $j$ and $\varepsilon >0$},
	\end{equation}
the numerator $P_{mn}$ of the $[m,n]$-Pad\'e approximant of $f_N(z) = \sum_{j=0}^N a_jz^j$ satisfies
	\begin{equation}
				\mathbb P \left( \frac{\log L(P_{mn})}{m} >\delta^4\right) <\delta.
	\end{equation} 
In particular, for $\mu$ the uniform measure on the unit circle we have
		\begin{align}
			\mathbb P \left( d_{\rm BL} (\nu_{P_{mn}}, \mu) > 40 \delta \right) < \delta.
		\end{align}
\end{theorem}

\begin{proof} Isotropicity implies that in applying Theorem \ref{thm:main-1} we can take $\gamma=1$, as $\mathbb E|a_j| \leq (\mathbb E|a_j|^2)^{1/2} =1$. 
If we take $\varepsilon = {\delta}/{(5K n)}$ we see that $b= \delta/(5n)$ in Theorem \ref{thm:main-1} will do. Then \eqref{eq:control-ET} becomes 
	\begin{equation}
			 	\mathbb P\left( \frac{\log L(P_{mn})}{m} 
				> 
				\frac{\log m}{m} + \frac{ 2 n \log (100 K  n^2/ \delta^2)}{m}\right) 
				<\delta.
	\end{equation}
Finally, observe that
there is a universal\footnote{For example, $C = e^\theta$ for $\theta \geq 5$ satisfying $e^\theta > \theta + 20$ will do.}  constant $C>0$, such that
for  $m \geq C \delta^{-4} n\log (e K n / \delta)$ we have
\begin{equation}
      \delta^4 > \frac{\log m}{m} + \frac{ 2 n \log (100 K n^2/ \delta^2)}{m},
\end{equation} 
so that for $m$ and $n$ in the same range
\begin{equation}\label{d4d}
				\mathbb P \left( \frac{\log L(P_{mn})}{m} >\delta^4\right) <\delta.
\end{equation} 
		For the last statement, combine \eqref{d4d}  with Proposition \ref{prop:disc}. 
\end{proof}		
		 The (static) estimate on the probability of failure in Theorem \ref{thm:main-1} is meaningful even for discrete distributions, provided that they have
		sufficient anti-concentration bounds, as the following example shows:			

\begin{example} \label{example} 
Fix $M$ a large positive integer and take the random variables $a_j$ to have common distribution supported 
uniformly on $\{0, \pm 1,\ldots, \pm M\}$, so that $\mathcal L(a_j, 1/2)  \asymp 1/M $ and $\mathbb E|a_j| \asymp M$. 
Then estimate \eqref{eq:control-ET} yields
		\begin{align*}
			\log L(P_{mn})  \leq \log m + C_1n \log(C_2nM),
		\end{align*}
and the clustering of the Pad\'e numerator at $|z| =1$ is measured by
		\begin{align*} 
    			\frac{\log m}{m} +\frac1m C_1n\log(C_2n M),
		\end{align*}
with probability greater than $1- cn/M$. For $n$ small compared to $M$, e.g.\  for $n\sim  \sqrt{M}$, clustering will be manifested 
whenever $m$ is sufficiently large, and in particular much larger than $n\log M$. 
\end{example}

\section{Asymptotic clustering of poles} \label{as}

We now turn our attention to the poles of Pad\'e approximants of random power series. 
For random $f_\omega(z) = \sum_{j = 0}^\infty a_j(\omega) z^j$ we shall apply on $1/f_\omega$, and separately for almost any 
$\omega$, Edrei's deterministic result for the clustering of Pad\'e numerators of meromorphic functions  
on circles defined by the distance of the poles from the orgin, \cite{E}.  
This will determine the clustering of the denominator $Q_{mn}$ via the elementary property
\begin{equation}
\begin{split}
     P_{nm}^{(1/f)} = Q_{mn}^{(f)},
\end{split}
\end{equation} 
where we use superscripts  to indicate the power series that the Pad\'e approximates,
see \cite[p.\ 216, Theorem 22(i)]{Gra} or \cite[p.\ 216]{Gil78}. 
For this we show first as Theorem \ref{radius} that under our assumptions the power series has almost always 
radius of convergence $1$ and that the unit circle is in fact a natural boundary. Then, after a few preliminary facts, we establish in Theorem \ref{roots} that
the random power series $f$ has zeros in the open unit disc $\mathbb{D}$ with high probability and gather information for the location of these zeros. 
We conclude by applying Edrei's result in the last subsection.

\subsection{Natural boundary}

Recall first the following:

\begin{definition} [Natural Boundary]
Let $r_f< \infty$ be the radius of convergence of the power series $f(z) = \sum_{j =0}^\infty a_j z^j$.
Then $|z| = r_f$ is the natural boundary of $f$ if $f$ cannot be extended to a holomorphic function through any arc of this circle.
\end{definition}

\begin{theorem} \label{radius}
Let $(a_j)_{j=0}^\infty$ be independent random variables with $\mathbb E|a_j| \leq \gamma < \infty$, 
and $\liminf_j {\mathcal L}(a_j , \varepsilon) < 1$ for some $\varepsilon > 0$. Then, the random power series 
$f(z) = \sum_{j =0}^\infty a_j z^j$ almost surely has radius of convergence  $1$ and the unit circle is almost surely a natural boundary for $f$.
\end{theorem}

\noindent {\it Proof.} Let $\tau = \liminf_j {\mathcal L} (a_j, \varepsilon)$. In particular, we have
	\[
		\limsup_j \mathbb P(|a_j| \geq \varepsilon) \geq 1-\tau > 0
		\quad \Longrightarrow 
		\quad
		\sum_{j=0}^\infty  \mathbb P(|a_j| \geq \varepsilon) = \infty.
	\] 
Then, the second Borel-Cantelli lemma yields that $\{ |a_j|^{1/j} \geq 1 \; \textrm {i.o.} \}$ is a sure event. That is, the radius of convergence
$r_f$ of $f$ satisfies $r_f \leq 1$ a.s. For the converse inequality we just observe that, by Markov's inequality, we have 
		\[
			\sum_{j=0}^\infty \mathbb P \left( |a_j| > (1+\delta)^j \right) \leq \gamma \sum_{j=0}^\infty (1+\delta)^{-j}< \infty,
		\]
and the result follows from the first Borel-Cantelli lemma.

To show that the unit circle is a natural boundary, recall that 
the Ryll-Nardzewksi theorem \cite{R-N} guarantees 
that\footnote{The existence of the natural boundary also follows from \cite{A}. We choose
to argue via the Ryll-Nardzewksi theorem as it is available, with a very direct proof, in \cite[p.\ 41]{K}.} if that were not the case
there would exist deterministic power series $g(z) = \sum_{j=0}^\infty b_j z^j$ such that $f-g$ would have radius of convergence $r_{f-g} > 1$ 
and the circle of this radius would be a natural boundary of $f-g$.
In particular, almost surely for all $z\in \mathbb C$ with $ |z|=1$ we would have  
\begin{equation}
      \sum_{j=0}^\infty (a_j(\omega) - b_j) z^j < \infty,
\end{equation} 
so that 
\begin{equation}
      a_j(\omega) - b_j \to 0,
\end{equation}
with probability 1. As almost sure convergence implies convergence in probability, this
contradicts the fact 
	\[
		\limsup_ j \mathbb P \left( \omega : | a_j(\omega) - b_j| \geq  \varepsilon \right) \geq 1- \tau > 0.
	\]
The proof is complete. \prend

\begin{remark} The almost sure existence of a natural  boundary appears often in the study of random power series. 
For symmetric random variables see \cite{R-N}. Our conditions in Theorem \ref{radius} are more relaxed than those in \cite{BS} where the $a_j$'s are independent, 
$\sup_{j} |a_j(\omega)| <M <\infty$ almost surely, and $\limsup_j \text{\rm Var}[a_j] >0$, see \cite[Theorem 6.1]{BS}. This follows from:
	\begin{fact}
		Let $\xi$ be a random variable with ${\rm Var}[\xi] \geq \theta >0$ and let $\mathbb E|\xi |^4 = \sigma_4<\infty$. Then
		\[
			{\mathcal L}\left( \xi,\sqrt{\theta/2} \right) < 1 - c\frac{\theta^2}{\sigma_4},
		\] where $c>0$ is a universal constant.
	\end{fact}	
We leave the details to the reader. 
However, we stress that the authors in \cite{BS} can show that in their case there is {\em strong} natural boundary, see \cite[p.\ 4905]{BS}.  
\end{remark}

\begin{remark}
	Note that the existence of natural boundary at $|z|=1$, in light of Hadamard's theorem \cite{Had1892}, 
	yields almost surely that for each $n$ we have
\begin{equation}
      \limsup_m \left| \det A_m^{(n)}\right|^{1/m} = 1.
\end{equation} However, under the assumptions of Theorem \ref{radius} one can show that the {\em limits} exist almost surely:

\begin{fact}
	For $(a_j)_{j=0}^\infty$ independent random variables with $\mathbb E |a_j| \leq \gamma < \infty$, 
	and ${\mathcal L}(a_j , \varepsilon) \leq K\varepsilon$ for all $\varepsilon>0$ and for some $K\geq 1$, it holds that almost surely, for each $n$,
	\begin{equation}
      		\lim_m \left| \det A_m^{(n)}\right|^{1/m} = 1.
	\end{equation}
\end{fact}
Indeed: We fix $n\geq 1$. From the arithmetic-geometric mean inequality
\begin{equation}
         \mathbb{E}\left | \det A_m^{(n)}\right|^{1/n} 
         \leq 
         \frac{\mathbb{E} \left \| A_m^{(n)}\right\|_{ \rm HS}}{\sqrt{n} }
         \leq \gamma n^{3/2} .
\end{equation}
From Markov's inequality
\begin{equation}
       \mathbb{P} \left( |\det A_m^{(n)}|^{1/n} >(1+\varepsilon)^{m/n} \right)
       \leq
       \frac{\mathbb{E}  | \det A_m^{(n)}|^{1/n}}{(1+\varepsilon)^{m/n}}
       \leq\frac{n}{(1+\varepsilon)^{m/n}},
\end{equation}  
and the first (direct) part of the Borel-Cantelli lemma gives that almost surely 
\begin{equation}
       \limsup_m |\det A_m^{(n)}|^{1/m} \leq 1 + \varepsilon.
\end{equation}          
On the other hand, again by Markov, 
\begin{equation}
      \mathbb{P} \left(  | \det A_m^{(n)}|^{-1/2n} > (1 + \varepsilon)^{m/2n} \right ) 
      \leq 
      \frac{\mathbb{E}  | \det A_m^{(n)}|^{-1/2n}}{(1+\varepsilon)^{m/2n}}
      \leq
      \frac{2Kn}{(1+\varepsilon)^{m/n}}, 
\end{equation}
where we have used Proposition \ref{prop:sb-det} and that for any random variable $\xi$
\begin{equation}
      \mathbb{P}(\xi<\varepsilon) \leq \gamma \varepsilon 
      \Rightarrow
      \mathbb{E}\left[ \xi^{-1/2}\right] \leq 1 + \gamma.
\end{equation}
Therefore, almost surely for any $\varepsilon >0$
\begin{equation}
      \limsup_m |\det A_m^{(n)}|^{-1/m} \leq 1+ \varepsilon
      \Rightarrow
      \liminf_m |\det A_m^{(n)}|^{1/m} \geq \frac1{1+ \varepsilon}.
\end{equation}

\end{remark}

\subsection{Random zeros in the unit disc}

The following lemma is standard. We shall use it to estimate averages of random polynomials on the unit circle.

\begin{lemma} \label{lem:dev}
Let $(\Sigma, \mu)$ be a probability space and let $(Y_\sigma)_{\sigma \in \Sigma}$ 
be a family of random variables with $\displaystyle Y: = \int Y_\sigma \, \mu(d\sigma)$ well defined. If
\begin{align*}
\sup_\sigma \mathbb P 
\left( 
Y_\sigma > t 
\right) \leq A e^{-at}, \quad 
\end{align*}
for some $A\geq 1$, $a>0$, and for all $t>0$, then 
\begin{align}\label{average of exp int}
\mathbb E [e^{\lambda Y}] \leq 3A, \quad 
\text{for all $0<\lambda \leq \frac{a}{2}$}. 
\end{align} In particular, 
\begin{align}\label{prob of int}
\mathbb P(Y > t) \leq 3A e^{-at/2},
\quad \text{for all $t>0$.}
\end{align} 
\end{lemma}

\noindent {\it Proof.} Let $\lambda>0$. Using Jensen's inequality with respect to $\sigma$ and Tonelli's theorem,
\[
\mathbb E [ e^{\lambda Y}] 
\leq 
\mathbb E \left[ \int e^{\lambda Y_\sigma} \, \mu(d\sigma) \right] 
= 
\int \mathbb E[e^{\lambda Y_\sigma} ] \, \mu(d\sigma).
\] For fixed $\sigma \in \Sigma$ we have
\[
\mathbb E \left[ e^{\lambda Y_\sigma} \right] \leq 1 + \mathbb E\left[ e^{\lambda Y_\sigma} \mathbf 1_{ \{Y_\sigma >0\} }\right] 
\leq 1+ \mathbb E\left[ e^{\lambda Y_\sigma^+} \right].
\] Finally, we have
\[
\mathbb E\left[ e^{\lambda Y_\sigma^+} \right] = 1 + \lambda \int_0^\infty e^{\lambda t} \mathbb P(Y_\sigma > t) \, dt \leq 1+ \lambda A \int_0^\infty e^{(\lambda -a) t} \, dt\leq 1+A,
\] since $\lambda -a \leq -\lambda$. Integrating this with respect to $\sigma$ gives \eqref{average of exp int}, since $A>1$. 
Markov's inequality for $\lambda = a/2$ gives \eqref{prob of int}.  \prend

\medskip

We next examine random polynomials $f_N(z) = \sum_{j=0}^N a_j z^j$ 
associated to isotropic vectors ${\bf a}=(a_j)_{j=0}^N$  in $\mathbb R^{N+1}$ for which there is $K\geq 1$ such that for  all  $\varepsilon>0$
\begin{equation} \label{anti unit}
      \mathbb P( |\langle {\bf a} , u \rangle| < \varepsilon)\leq K\varepsilon, 
      \quad \text{for all $u$ with $\|u\|_2=1$},
\end{equation}
and we set 
\begin{equation}
      X_z : = |f_N(z)|^2.
\end{equation}
Note that $X_z$ is also given by 
	\[
	X_z  = \|V_z({\bf a})\|_2^2,
	\]
 for $V_z: \mathbb R^{N+1}\to \mathbb R^2$ the linear map with 
	\[
	V_z(e_j) = \left( |z|^j \cos (j\theta) , |z|^j \sin(j\theta)  \right), \quad \theta = {\rm Arg}(z),
	\] 
for $j=0,1,\dots, N$ and $(e_j)_{j=0}^N$ denotes the standard basis of $\mathbb R^{N+1}$. 
The isotropicity of $\bf a$ implies\footnote{Recall that $\|A\|_{\rm HS}$ stands for the Hilbert-Schmidt norm of the matrix $A$, i.e., $\|A\|_{\rm HS}^2 = \sum_{i,j} |a_{ij}|^2$.}
	\begin{align} \label{three means}
		\mathbb E [X_z] = \|V_z\|_{\rm HS}^2 = \sum_{k=0}^n |z|^{2k}.
	\end{align}
As we intend to apply Lemma \ref{lem:dev} to $\log (X_z/\mathbb E [X_z] )$, we first need the following:
\begin{lemma}\label{f:4-5}
For all $z\in \mathbb C$ and all $t>0$ it holds that
\begin{align} \label{eq:4-1}
		\mathbb P\left( \left|\, \log X_z - \log \mathbb E [X_z] \,\right| > t \right) \leq 3K e^{-t/2}.
\end{align} 
\end{lemma}
\begin{proof}
Fix $z\in \mathbb C$. By \eqref{three means} we have
	\begin{align*}
		\mathbb P\left( \left| \log X_z - \log\mathbb E [X_z] \right| > t \right) &=
					\mathbb P\left( X_z > e^t \mathbb E[X_z] \right) + \mathbb P \left(X_z \leq e^{-t} \|V_z\|_{\rm HS}^2 \right).
	\end{align*} 
The first probability is bounded by $e^{-t}$ using Markov's inequality. 
	
For the second,
notice that either $\| V_z^\ast(e_1)\|_2^2\geq \|V_z\|_{\rm HS}^2/2$ or $\| V_z^\ast (e_2) \|_2^2\geq \| V_z \|_{\rm HS}^2/2$. 
Without loss of generality assume the former. 
By H\"older's inequality and for 
 the unit vector $u_z: = V_z^\ast(e_1) / \| V_z^\ast(e_1)\|_2$ we have 
 \begin{align*}
		\mathbb P\left( X_z \leq e^{-t} \|V_z\|_{\rm HS}^2 \right) &\leq \mathbb P\left( |\langle {\bf a}, V_z^\ast(e_1)\rangle|^2 \leq e^{-t} \|V_z\|_{\rm HS}^2\right) \\	
				&\leq \mathbb P \left ( |\langle {\bf a}, u_z \rangle| \leq e^{-t/2} \frac{\|V_z\|_{\rm HS} }{\|(V_z)^\ast(e_1)\|_2} \right) \\
						& \leq \mathbb P\left( |\langle {\bf a}, u_z \rangle| \leq e^{-t/2} \sqrt{2}\right) \\
								& \leq e^{-t/2}K\sqrt{2}. \qedhere
	\end{align*}  
\end{proof}

\begin{proposition} \label{prop:hole} 
	Let ${\bf a}=(a_j)_{j=0}^N$ be an isotropic random vector in $\mathbb R^{N+1}$ 
	satisfying \eqref{anti unit}. Then for all $r,t>0$ the random polynomial $f_N(z) = \sum_{j=0}^N a_j z^j$ satisfies 	
	\begin{align} \label{eq:main-4-1}
			\mathbb P\left( \left| \frac{1}{2\pi}\int_0^{2\pi} \log |f_N(re^{i\theta})| \, d\theta -\frac{1}{2}\log \rho_N(r) - \log |a_0| \right| > t \right) \leq CK e^{-ct}, 
	\end{align} where $\rho_N(r)=\sum_{k=0}^N r^{2k}$. 
In particular,
	\begin{align}\label{inparticular}
			\mathbb P( \nu_{f_N}(\mathbb D) >0 )\geq 1 - CKN^{-c}.
	\end{align}
\end{proposition}

\begin{proof} 
For any fixed $r>0$ we set
\begin{equation} 
      X_\theta := X_{re^{i\theta}}, \quad
      \quad 
      Y_\theta: = \log \left( \frac{X_\theta}{\mathbb E[X_\theta] }\right),
      \quad
	 Y:= \frac{1}{2\pi} \int_0^{2\pi} Y_\theta  \, d\theta.
\end{equation}
In view of Lemma \ref{f:4-5}, apply Lemma \ref{lem:dev}  for $(\pm Y_\theta)_{\theta\in [0,2\pi]}$ to have
	\begin{align} \label{eq:4-6}
		\mathbb P \left( |Y| > t \right) \leq 18 K e^{ - t/4}, \quad t>0.
	\end{align} Since 
	\begin{align}
		\mathbb P( \left| \log |a_0| \right| > s) \leq e^{-2s} + \mathbb P(|a_0| < e^{-s}) \leq 2K e^{-s}, \quad s>0,
	\end{align} we get
	\begin{align}
		\mathbb P \left( \left| Y + 2\log |a_0| \right | > 2t\right) \leq 18K e^{-t/4} + 2Ke^{-t/2} \leq 20 K e^{-t/4},
	\end{align} which proves \eqref{eq:main-4-1}.
\medskip
	
For \eqref{inparticular}  we use Jensen's formula from Proposition \ref{JensenForm}:
	\[
		\frac{1}{2\pi} \int_0^{2\pi} \log |f_N(e^{i\theta})| \, d\theta - \log |a_0| 
		= 
		N \nu_{f_N}(\mathbb D) 
	\] Applying \eqref{eq:main-4-1} for $r=1$ and $t= \frac{1}{4} \log (N+1)$ we find that 
	\[
		\mathbb P\left( \nu_{f_N}(\mathbb D)  \geq \frac{\log (N+1)}{4N}\right) 
		>
		1-20 Ke^{-\frac{1}{16}\log (N+1)},
	\] 
as asserted. \end{proof}

\begin{notation}
${\frak N}_{f}(r)$ will denote the number of zeros of $f$, counted with multiplicity, in the open disc of radius $r$ centered at the origin.

 For $s= 0, 1, 2, \ldots$ let $R_s(f)$ denote the largest radius $r>0$ for which the function $f$ has no more than $s$ zeros in the open disc of radius $r$ centered at the origin. 
\end{notation}
Recall here that, by a straightforward calculation, if we arrange the zeros of a holomorphic function in increasing 
distance from the origin $|z_1|\leq |z_2|\leq \ldots$, then
\begin{equation} \label{numberfromintegral}
      \int_0^r \frac{{\frak N}_{f}(t)}{t} dt
      =
      \sum_{j =1}^{{\frak N}_{f}(r)} \log \frac{r}{|z_j|}.
\end{equation}
We shall need the following:
\begin{lemma}	\label{lem:aux-4-1}
	Let $f: \mathbb D\to \mathbb C$ be a holomorphic function with $f(0)\neq 0$. If $f_n:\mathbb D \to \mathbb C$ is a sequence
	of holomorphic functions which converges uniformly on compacts to $f$, then for any $0<r<1$ we have
		\begin{align*}
			\lim_{n\to \infty} \int_0^{2\pi} \log |f_n(re^{i\theta})| \, d\theta = \int_0^{2\pi} \log |f(re^{i\theta})| \,d\theta.
		\end{align*}
\end{lemma}
\begin{proof} As each $f_N$ is holomorphic 
on a domain that includes $|z|\leq r$ then, regardless of zeros on the circle $|z|=r$ \cite[Theorem 4.8.3]{AN}, Jensen's formula
yields
\begin{equation}
    \frac{1}{2\pi}\int_0^{2\pi} \log|f_n(re^{i\theta})| \, d \theta
    =
    \log|f(0)| + \sum_{j=1}^{{\frak N}_{f_n}(r)} \log\frac{|z_j^{(n)}|}{r},
\end{equation}
where ${\frak N}_{f_n}(r)$ is the number of zeros of $f_n$ in $|z|<r$ and $z_j^{(n)}$ denote the zeros of $f_n$.
Also,
\begin{equation}
    \frac{1}{2\pi}\int_0^{2\pi} \log|f(re^{i\theta})| \, d \theta
    =
    \log|f_n(0)| + \sum_{j=1}^{{\frak N}_f(r)} \log\frac{|z_j|}{r}.
\end{equation}
Given any $\varepsilon>0$, by the theorem of Hurwitz \cite[Theorem 5.1.3]{AN}, after a certain $n$ we have ${\frak N}_f(r) = {\frak N}_{f_n}(r)$,
and for any $z_j$ there is unique $z_l^{(n)}$, call it $z ^{(n)}(j)$, with $|z_j - z^{(n)}(j)| < \varepsilon$, so that 
\begin{equation}
    \left|\sum_{j=1}^{{\frak N}_f(r)} \log\frac{|z_j|}{r}
    -
    \sum_{j=1}^{{\frak N}_{f_n}(r)} \log\frac{|z_j^{(n)}|}{r}\right|
    \leq
    \sum_{j=1}^{{\frak N}_f(r)} 
    \left|\log|z_j| - \log|z^{(n)}(j)|\right| ,
\end{equation}
which can be made arbitrarily small.
\end{proof}
We now prove the main result of this section, Theorem \ref{roots}, which should be compared to \cite[p.\ 180, Theorem 1]{K} that considers only Gaussian $a_j$'s and also uses Jensen's formula. 
\begin{theorem} \label{roots}
Let $(a_j)_{j=0}^\infty$ be a sequence of independent random variables on a probability space $(\Omega, \Sigma, \mathbb P)$
with $\mathbb E[a_j]=0$, $\mathbb E[a_j^2]=1$, and ${\mathcal L}(a_j , \varepsilon) \leq K\varepsilon$ for all $\varepsilon>0$ and all $j$. 
Then, for the random power series $f_\omega (z) = \sum_{j=0}^\infty a_j(\omega) z^j$ the following hold true:
	\begin{enumerate}
		\item We have
			\begin{align} \label{eq:limlaw}
				\liminf_{r\uparrow 1}
				\left|
				 \frac{1}{\log(1-r^2)} \int_0^r \frac{{\frak N}_f(t)}{t} \, dt
				  +\frac{1}{2}\right| =0, \quad a.s.
			\end{align}
		\item The zero-set of $f$ in $\mathbb D=\{z\in \mathbb C : |z| <1\}$ is infinite a.s., $0<R_s(f)<1$ for $s=0,1,\ldots$, and $R_s(f) \uparrow 1$. 
		
Moreover, for almost every $\omega$, we have
\begin{equation} \label{1-R>R}
       1 -R_s(f_\omega) \geq R_0(f_\omega)^{cs} ,
\end{equation} for all sufficiently large $s$, and
\begin{equation}\label{bigO}
      1-R_s(f_\omega) = O \left( \dfrac{\log s}{s} \right), \quad s\to \infty. 
\end{equation}
			
	\end{enumerate}
\end{theorem}

\noindent {\it Proof.}  (1) Let $f_N(\omega,z) = \sum_{j=0}^N a_j(\omega) z^j$ be the $N$-th partial sum of $f_\omega$. Then, the random 
vector ${\bf a}_N=(a_0,\ldots,a_N)$ is isotropic and has independent coordinates with ${\mathcal L}(a_j, \varepsilon) \leq K\varepsilon$ for all 
$\varepsilon>0$. Therefore, from \cite[Corollary 1.4]{RV-sb}, we conclude that for all $\|u\|_2=1$ the anti-concentration estimate 
\begin{equation}
      \mathbb P(|\langle {\bf a}_N, u\rangle| \leq \varepsilon ) \leq CK\varepsilon
\end{equation}
holds for all $\varepsilon >0$. 
Then if $t_{N,r,\delta}=\delta \log \rho_N(r), \; 0 < r <1$, $\delta>0$, and 
\begin{align}
		{\mathcal E}_N(r,\delta) := \left\{ \left| \frac{1}{2\pi} \int_0^{2\pi} \log|f_N(\omega, re^{i\theta} ) | \, d\theta - \frac{1}{2} \log \rho_N(r) -\log |a_0| \right | > t_{N,r,\delta} \right \},
	\end{align}
Proposition \ref{prop:hole}, applied for $t=t_{N,r,\delta}$, shows that 
	\begin{align} \label{eq:4-9}
		\mathbb P({\mathcal E}_N(r,\delta)) \leq CKe^{-t_{N,r,\delta}/4} = CK e^{-c \delta \log \rho_N(r) }.
	\end{align}
Now if
	\[
		{\mathcal E}_\infty(r,\delta) := \left\{ \left| \frac{1}{2\pi} \int_0^{2\pi} \log|f_\omega(re^{i\theta})| \, d\theta  -\log |a_0| +\frac{1}{2} \log (1-r^2) \right| > -\delta \log(1-r^2) \right \},
	\] 
	Lemma \ref{lem:aux-4-1} implies that ${\mathcal E}_\infty(r,\delta) \subset \liminf_N \mathcal E_N(r,\delta)$, and Fatou's lemma implies that 
	\[
	\mathbb P \left( \liminf_N \mathcal E_N(r,\delta) \right) \leq \liminf_N \mathbb P(\mathcal E_N(r,\delta)) \stackrel{\eqref{eq:4-9}}\leq CK(1-r^2)^{c \delta}.
	\] Having derived
		\begin{align} 
			\mathbb P (\mathcal E_\infty(r,\delta)) \leq CK (1-r^2)^{c\delta}, \quad 0<r<1, \; \delta >0,
		\end{align}
Jensen's formula for $f$ shows that 
	\begin{align} \label{eq:aux-dev-2}
		{\mathcal E}_\infty(r,\delta) = \left\{ \left| \frac{1}{2} + \frac{1}{\log(1-r^2)} \int_0^r \frac{ {\frak N}_f(t)}{t} \, dt \right| > \delta \right\}	.
	\end{align}
Now the choice $r=r_{m,\delta} := \sqrt{1 - m^{-2/(c\delta) }}$, $m=1,2,\ldots$ yields that
	\[
		\sum_m \mathbb P(\mathcal E_\infty(r_{m,\delta}, \delta)) \leq CK \sum_m (1-r_{m,\delta}^2)^{c\delta} \leq CK \sum_m \frac{1}{m^2} < \infty.
	\] Hence, the first Borel-Cantelli lemma in turn shows that the set
	\begin{align} \label{eq:null-1}
	\mathcal N_\delta:= \limsup_m \mathcal E_\infty (r_{m,\delta}, \delta), 
	\end{align} is null for each $\delta >0$. Further, let $\mathcal N : =  \bigcup_{q \in \mathbb Q^+} \mathcal N_q$. Clearly, $\mathcal N$ is a null set
	and for every $\omega\notin \mathcal N$ it holds that: for any $\varepsilon > 0$ there is subsequence of $r$'s (of the form $r_{m,1/n}$ for $1/n<\varepsilon$) 
	such that for that subsequence we eventually have 
	\begin{equation}
       \left| \frac{1}{2} + \frac{1}{\log(1-r^2)} \int_0^r \frac{ {\frak N}_f(t)}{t} \, dt \right| <\varepsilon, 
\end{equation}
for any $\varepsilon > 0$. 
In particular, $\displaystyle \liminf_r \left| \frac{1}{2} + \dfrac{1}{\log(1-r^2)} \int_0^r \dfrac{ {\frak N}_f(t)}{t} \, dt \right| < \varepsilon$ for all $\varepsilon > 0$.

\smallskip

\noindent (2) Part(1) implies that for a.e. $\omega\in \Omega$ the zero-set of $f_\omega$ in $\mathbb D$ is infinite, therefore we have strict inequality $R_s(f_\omega)<1$ for all $s$. 
Since $a_0\neq 0$ a.s. we have that $R_0(f)>0$ a.s.
By definition, $(R_s)_{s=0}^\infty$ is a non-decreasing sequence of numbers. Then as
\begin{equation}
      \sum_{j =1}^{{\frak N}_{f}(r)} 
      \log \frac{r}{|z_j|} 
      \leq 
      {\frak N}_f(r)\log\frac{1}{R_0},
\end{equation}
and $R_{ {\frak N}_f(r)} \geq  r$, \eqref{numberfromintegral} and \eqref{eq:limlaw}  show that $R_s\to 1$ as $s\to \infty$. 

A more careful analysis of the argument leading up to \eqref{eq:limlaw} reveals the announced rates of convergence. 
Indeed; for \eqref{1-R>R},  consider $\mathcal N_{1/3}$ from \eqref{eq:null-1}. Then, for $\omega \notin \mathcal N_{1/3}$ there exists 
$m_0(\omega)\in \mathbb N$ such that for all $m\geq m_0$ we have
\begin{equation} \label{double}
		\frac{1}{6}\log[ (1-r_m^2)^{-1}] \leq \int_0^{r_m} \frac{ {\frak N}_{f}(t) }{t} \, dt \leq \frac{5}{6} \log[(1-r_m^2)^{-1}],
\end{equation}
where $r_m\equiv r_{m,1/3}= \sqrt{1 - m^{-6/c}}$. It follows from the first inequality in \eqref{double} that
	\[
		\frak N_{f} (r_m) \log (1/R_0) \geq \int_{R_0}^{r_m} \frac{ \frak{N}_{f}(t) }{t} \, dt  = \int_{0}^{r_m} \frac{ \frak N_{f}(t)}{t} \, dt 	\geq c \log m,
	\]
 for all $m\geq m_0$. Then, for any $s >  \dfrac{c\log m_0}{\log(1/R_0)}$, and for the smallest $m$ so that $s\leq \dfrac{c\log m}{\log(1/R_0)}$ 
 we have $s \leq  \frak N_{f} (r_m) $. By the definition of $R_s$, we also have  $R_s\leq r_m$.  
 Finally use $r_m \leq 1 -m^{-c'}$ to conclude \eqref{1-R>R}. 
 
 For \eqref{bigO},  first note that $\frak N_f(t) \geq  s$ for all $t>R_s$. Therefore, using the second inequality in \eqref{double},
	\begin{align} \label{eq:4-a-1}
		\frac{s}{2} \log \frac{1}{R_s} \leq \int_{R_s}^{\sqrt{R_s}} \frac { \frak N_f(t)}{t} \, dt \leq \int_0^{r_{q_s+1}} \frac{ \frak N_f(t)}{t} \, dt \leq C \log q_s,
	\end{align} where $q_s:=\min \{m \mid r_{m+1}^2 \geq R_s\}$. By the definition of $q_s$ we have 
	\begin{align} \label{eq:4-a-2}
		R_s > r_{q_s}^2 = 1 - q_s^{-c} \quad \Longrightarrow \quad c \log q_s  <  \log \frac{1}{1-R_s}.
	\end{align} Combining \eqref{eq:4-a-1} with \eqref{eq:4-a-2} we obtain
	\begin{align*}
		 cs \log \frac{1}{R_s} \leq \log \frac{1}{1-R_s}.
	\end{align*} 
Since $1/2< R_s < 1$ for all sufficiently large $s$, we conclude\footnote{Set $y:=(1-R_s)^{-1} \in (2,\infty)$. Then, 
\[ 
\log y \geq cs \log \left(1 + \frac{1}{y-1} \right) \geq   cs \frac{1}{y} \quad  \Longrightarrow  \quad y \log y \geq cs.
\] It is easy to check that for $h(y)=y\log y, \; y > 2$ we have $h^{-1}(u) \asymp \dfrac{u}{\log u}$ from which the 
\eqref{bigO} follows.} 
that $1-R_s \leq C \log s/ s$, as claimed. \prend

\subsection{Clustering of poles}

Recall first the following:

\begin{definition} [Radius of meromorphicity]
The radius of meromorphicity of a power series $g(z) =\sum_{j= 0}^\infty b_j z^j$ is the largest 
radius $r>0$ with the property that there is polynomial $p$ of finite degree with $p(z) g(z)$ holomorphic on $|z|<r$.
\end{definition}

In Theorem \ref{radius} we showed that, for $a_j$ satisfying the assumptions there,  $f (z) = \sum_{j=0}^\infty a_j z^j$ has natural boundary. The following translates this into $1/f$:
\begin{lemma} \label{1/f}
If $f$ has natural boundary at $|z|=1$ then  $1/f$ has radius of meromorphicity $1$.
\end{lemma}

\begin{proof}
As $f$ is holomorphic on $|z|<1$, $1/f$ is meromorphic there, cf. \cite[Remarks 4.2.5]{AN}, therefore the radius of meromorphicity of $1/f$ is at least $1$. 
If it were more than $1$ there would exist polynomial $P(z)$ such that $P(z)/f(z)$ would be holomorphic with radius of convergence $r>1$ and $f(z)/P(z)$ 
would be meromorphic on $|z|<r$. In particular, there would exist neighborhoods of points on the unit circle where $f/P$, and therefore $f$, would be holomorphic. 
Then $|z|=1$ would not be the natural boundary of $f$.
\end{proof}
 
Now recall Edrei's deterministic result \cite[Theorem 1]{E} that we shall apply for $1/f$:

\begin{theorem}[Edrei] \label{EMain}
Let $g(z) = \sum_{j=0}^\infty b_j z^j$ be the power series expansion at $0$ of a function with radius of meromorphicity $0<\tau <\infty$. Then for any $n \geq 0$ there exists $m$-subsequence of Pad\'e numerators $P^{(g)}_{mn}$ of $g$ whose zeros cluster  uniformly around the circle of radius $\sigma_n$, defined as the largest radius of a disc centered at the origin containing no more than $n$ poles of $f$.
\end{theorem}
In the notation of Definition \ref{Clustering defintion}, this means that there is subsequence $m_l$ such that 
\begin{equation}
       \nu_{P_{m_ln}}(R(\sigma_n, \rho) )\to 1, 
       \quad 
       \nu_{P_{m_ln}}(S(\theta, \phi)) \to \frac{\phi - \theta}{2\pi}, 
       \quad  l \to \infty,
\end{equation}
for all $\rho>0$, and all $0 \leq \theta < \phi < 2\pi$.

\begin{theorem} [Clustering of poles for Pad\'e approximant] \label{polecluster}
Let $f$ be as in Theorem \ref{roots} and $R_m(f_\omega)< 1$ the radius of the largest disc that contains no more than $m$ zeros of $f_\omega(z)$, also defined in Theorem \ref{roots}. 
Then for almost all $\omega$ the zeros of the Pad\'e denominator $Q_{mn}^{(f_\omega)}$ cluster uniformly, up to subsequence as $n \to \infty$, at $|z| = R_m(f_\omega)<1$, with $1 - R_m(f_\omega) = O(\log m /m)$.
\end{theorem}
\begin{proof}
By Lemma \ref{1/f}, the radius of meromorphicity of $1/f$ is almost surely $1$. And the largest radius that contains no more than $m$ poles of $1/f$ is $R_m$. 
Therefore, by Theorem \ref{EMain}, the Pad\'e numerators $P_{ns}^{(1/f)}$ of $1/f$ cluster uniformly, 
as $n \to \infty$ at $|z|=R_m$. It remains to notice that the Pad\'e numerator $P_{nm}^{(1/f)}$ of $1/f$ is the same as the Pad\'e denominator $Q_{mn}^{(f)}$ of $f$, see \cite[p.\ 216, Theorem 22(i)]{Gra}. Finally use Theorem \ref{roots}.
\end{proof}

\section{Logarithmically concave data} \label{LogConcave}

In this section we return to the clustering of zeros of the Pad\'e approximant initiated in Section \ref{non-as} and examine the case when the $a_j$'s in $\sum_{j=1}^N a_j z^j$ are not necessarily independent. Our focus is on the fairly large class of (isotropic) log-concave vectors. 

Recall that a 
random vector $X$ on $\mathbb R^n$ is log-concave if for any two compact sets $K,L$, and for any $0<\lambda<1$, we have
	\begin{align} \label{eq:log-def}
		\mathbb P( X\in (1-\lambda) K + \lambda L) \geq [\mathbb P(X\in K)]^{1-\lambda} [\mathbb P(X\in L)|^\lambda.
	\end{align}
It is well known, see e.g., \cite{Bor75}, \cite[Theorem 2.1.2]{BGVV}, that if $X$ is a log-concave random vector on $\mathbb R^n$, non-degenerate in the sense that $\mathbb P(X \in H)<1$ for every hyperplane $H$, then the distribution of $X$ 
has density $f_X$, log-concave on its support, i.e.
	\begin{align}
		f_X((1-\lambda)x + \lambda y) \geq f_X(x)^{1-\lambda} f_X(y)^\lambda, \quad \text{for all}\ x,y, \quad 0<\lambda <1.
	\end{align} 
See \cite{BGVV} for background material on log-concave distributions and their geometric properties.

From Section  \ref{non-as}, notice that in establishing non-asymptotic clustering of roots 
the independence of the entries of the data vector ${\bf a} = (a_0,\ldots, a_N)$ is barely used. The properties required to establish the counterpart 
of Theorem \ref{thm:adv} for log-concave vectors ${\bf a} \in \mathbb R^{N+1}$ can be summarized as follows:

\begin{itemize}
	\item [P1:] Upper bound for ${\mathcal L} (a_j ,\varepsilon)$, $\varepsilon>0$.
	
	\item [P2:] Upper bound for $\mathbb P( \| P_\sigma ({\bf a}) \|_1 > t )$, $t>0$, where $\sigma \subset \{0,1,\ldots,N\}$ and $P_\sigma$ stands for the coordinate
	projection onto $\mathbb R^\sigma = {\rm span}\{e_i : i\in \sigma\}$.
	
	\item [P3:] Upper bound for $\mathbb P( |\det A_m^{(n)}|^{1/n} < \varepsilon)$, $\varepsilon>0$, for $A_m^{(n)}$ the  $n \times n$ Toeplitz matrix associated to $\bf a$ as in \eqref{assocmatr}:
	\begin{equation} 
     A_m^{(n)} = \begin{bmatrix} 
			 a_{m} & a_{m-1} & \ldots & a_{m- n+1} \\
			 a_{m+1} & a_{m} & \ldots & a_{m-n+2} \\
			\vdots & \vdots & \ddots  & \vdots \\
		 a_{m+n -1} & a_{m+n -2} & \ldots & a_{m}
		\end{bmatrix}.
\end{equation}

\end{itemize}

Properties P1 and P2 will follow easily is subsections \ref{logconanti} and \ref{logconlarge}, respectively. We shall rely on the following fact which can be directly verified from the definitions:

\begin{fact} \label{fact:hered}
	Let $X$ be a log-concave vector on $\mathbb R^n$ and let $T:\mathbb R^n \to \mathbb R^k$ be a linear mapping. Then, the vector $TX$ is 
	also log-concave. In particular, if $E$ is a $k$-dimensional subspace of $\mathbb R^n$, the marginal $P_E X$ is log-concave, where
	$P_E$ denotes the orthogonal projection onto $E$. Furthermore, if $X$ is isotropic, so is $P_EX$. 
\end{fact}

The more involved property P3  is established in subsection \ref{logcontoep}. 

\subsection{Anti-concentration for coordinates of log-concave vectors} \label{logconanti}

We recall first that 1-dimensional marginals of an isotropic, 
log-concave vector $X$ satisfy anti-concentration bounds of the following form:

\begin{lemma}
	Let $X$ be an isotropic, log-concave random vector on $\mathbb R^n$. Then, for any $\theta \in \mathbb R^n$ with $\|\theta\|_2=1$ we have 
	\begin{align}
		{\mathcal L} (\langle X, \theta \rangle, \varepsilon) \leq C\varepsilon,
	\end{align}
	for all $\varepsilon>0$, where $C>0$ is a universal constant. 
\end{lemma}

\noindent {\it Proof.} Fix $\theta\in \mathbb R^n$ with $\|\theta \|_2=1$. Then, in view of Fact \ref{fact:hered},
the random variable $P_\theta X:= \langle X, \theta \rangle$ has mean zero, variance 1, 
and is log-concave. Let $f_\theta$ be its density function. It is well known, see e.g. \cite{Frad}, \cite[Theorem 2.2.2]{BGVV}, 
that $\|f_\theta\|_\infty \leq e f_\theta(0)$. Also, \cite[Theorem 2.3.3]{BGVV} implies that $f_\theta(0) \leq C$. Thus, for any $a\in \mathbb R$ we may write
	\begin{align}
		\mathbb P( |\langle X,\theta \rangle - a| <\varepsilon) = \int_{a- \varepsilon}^{a+\varepsilon} f_\theta(t) \, dt \leq 2\varepsilon \|f_\theta\|_\infty,
	\end{align}
and the assertion follows. \prend


\subsection{Large deviation estimates for projections of log-concave vectors} \label{logconlarge}

As noticed, the projections of isotropic (log-concave) vectors are also isotropic (and log-concave).
Therefore, we may simply bound
	\begin{align} \label{eq:w-dev}
		\mathbb P \left( \|P_\sigma X\|_1 > t \right) \leq \frac{|\sigma|}{t}, \quad \forall t>0,
	\end{align} 
by using Markov's inequality and the basic fact that $\mathbb E|X_i| \leq 1$.

\begin{note} For isotropic log-concave vectors stronger large deviation estimates are available: The celebrated 
deviation inequality of Paouris \cite{Pa-dev} implies
	\begin{align*}
		\mathbb P \left( \|P_\sigma X\|_2 \geq Ct \sqrt{|\sigma|} \right) \leq C\exp(-t\sqrt{|\sigma|}), \quad \forall \, t\geq 1,
	\end{align*} 
where $C>0$ is a universal constant. (See also \cite{LS} for kindred estimates with respect to other $\ell_p$-norms.) Nonetheless,
since this event will later be unified with a less rare event, the weaker probability of the latter will persist. Hence, we have argued with the trivial bound
\eqref{eq:w-dev} which is adequate for our purposes.
\end{note}


\subsection{Determinant of Toeplitz matrices with log-concave symbol} \label{logcontoep}

The rest of the section is devoted to proving the following estimate:

\begin{theorem} \label{thm:4-main}
	Let ${\bf a} = (a_0, \ldots, a_{2n-2})$ be an isotropic, log-concave random vector in $\mathbb R^N$, $N:=2n-1$, and let $T({\bf a})$ be the  $n\times n$ Toeplitz matrix
\begin{equation}
       T({\bf a}) =    \begin{bmatrix} 
			 a_{n-1} & a_{n-2} & \ldots & a_{0} \\
			 a_{n} & a_{n-1} & \ldots & a_{1} \\
			\vdots & \vdots & \ddots  & \vdots \\
		 a_{2n -2} & a_{2n -3} & \ldots & a_{n-1}
		\end{bmatrix}.      
\end{equation} 
Then the 
	following anti-concetration bound holds: 
		\begin{align}
			\mathbb P( |\det T({\bf a})|^{1/n} \leq \varepsilon) \leq Cn^c \varepsilon,
		\end{align} for all $\varepsilon>0$, where $C,c>0$ are universal constants.
\end{theorem}

This result will follow from \cite[Theorem 8]{CW} once we have established a lower bound for $\mathbb E|\det T({\bf a})|$. To this end, we have the following:

\begin{proposition} \label{prop:4-det}
	Let ${\bf a}$ be a vector as above and let $g\sim N(0,I_N)$. The following estimate holds:
	\begin{align}
		\left(\mathbb E|\det T({\bf a})| \right)^{1/n} \geq cn^{-3/2} \left(\mathbb E|\det T(g)| \right)^{1/n},
	\end{align} where $c>0$ is a universal constant.
\end{proposition}

\noindent {\it Proof.} Integration in polar coordinates along with the fact that $x \mapsto |\det T(x)|$ is $n$-homogeneous yields
	\begin{align}
		\mathbb E |\det T({\bf a})| &= N \omega_N \int_0^\infty \int_{S^{N-1}} r^{N+ n -1} |\det T(\theta)| f_{\bf a}(r\theta) \, d\sigma(\theta) \, dr \nonumber \\
			&= \frac{N\omega_N}{N+n} f(0) \int_{S^{N-1}} |\det T(\theta)| \rho_{K_{N+n}(f_{\bf a})}^{N+n}(\theta) \, d\sigma(\theta), \label{eq:4-4}
	\end{align} where $\rho_{K_p(f_{\bf a})}$ is the radial function of K. Ball's body associated with $f_{\bf a}$, see \cite[Section 3]{Pa-sb}, \cite[Section 2.5]{BGVV} for the definition.
Applying the same formula for ${\bf a}$ replaced by $g$ we find
\begin{align} \label{eq:4-4b}
	\mathbb E |\det T(g)| = \frac{N\omega_N}{N+n} \frac{2^{n/2} \Gamma(\frac{N+n}{2} +1)}{\pi^{N/2}}\int_{S^{N-1}}  |\det T(\theta)| \, d\sigma(\theta) .
\end{align}
In view of \cite[Proposition 2.5.7]{BGVV}, \cite[Eqt. (3.11)]{Pa-sb} we have that 
\begin{align}
	\rho_{K_{N+n}}(\theta) \geq e^{\frac{N}{N+1} - \frac{N}{N+n}} \rho_{K_{N+1}}(\theta)
\end{align} for all $\theta \in S^{N-1}$.  
Hence, we may lower bound \eqref{eq:4-4} by
\begin{align} \label{eq:4-5}
	\mathbb E |\det T({\bf a})| \geq \frac{N\omega_N}{N+n} f(0) e^{cn} |K_{N+1}(f_{\bf a})|^{\frac{N+n}{N}} \int_{S^{N-1}} |\det T(\theta)| \rho_{\overline{K_{N+1}(f_{\bf a})}}^{N+n}(\theta) \, d\sigma(\theta),
\end{align} where $\overline A$ denotes the homothetic image of $A$ of volume 1, i.e., ${\overline A} := |A|^{-1/N} A$.

Next, since $f_{\bf a}$ is isotropic it follows that $\overline {K_{N+1}(f_{\bf a})}$ is almost isotropic, hence $\rho_{\overline{K_{N+1} } }(\theta) \geq c_1f_{\bf a}(0)^{1/N}$ for all $\theta \in S^{N-1}$. 
Taking into account this, \eqref{eq:4-4b}, and \eqref{eq:4-5} we may write
\begin{align}
	\mathbb E|\det T({\bf a})| \geq \frac{c_2^n}{\Gamma(\frac{N+n}{2}+1)} f_{\bf a}(0)^{2+\frac{n}{N}} |K_{N+1}(f_{\bf a})|^{\frac{N+n}{N}} \mathbb E |\det T(g)|.
\end{align}
From \cite[Proposition 2.5.8]{BGVV} we have that 
\begin{align}
	|K_{N+1}(f_{\bf a})| \geq c_3/f_{\bf a}(0).
\end{align} 
Plugging this estimate into the previous inequality we derive the following:
\begin{align}
	\mathbb E|\det T({\bf a})| \geq \frac{c_4^n}{(c_5n)^{3n/2}} f_{\bf a}(0) \mathbb E|\det T(g)|.
\end{align}
Finally, since $f_{\bf a}$ is isotropic one has $f_{\bf a}(0)^{1/N} \geq c_6>0$, see \cite[Proposition 2.3.12]{BGVV}, thus the result follows. \prend

\medskip

\noindent {\it Proof of Theorem \ref{thm:4-main}.} Viewing $\det T({\bf a})$ as a (homogeneous) polynomial on ${\bf a}$ of degree $n$,
and employing the anti-concentration estimate due to Carbery and Wright \cite[Theorem 8]{CW} we find
	\begin{align}
		\left(\mathbb E|\det T({\bf a})| \right)^{1/n}\cdot \mathbb P( |\det T({\bf a})|^{1/n} \leq \varepsilon) \leq C n \varepsilon, \quad \varepsilon>0.
	\end{align}
On the other hand Proposition \ref{prop:4-det}, in conjunction with Proposition \ref{prop:sb-det}, yields 
\begin{align*}
	(\mathbb E |\det T({\bf a})|)^{1/n} \geq cn^{-3/2} \left(\mathbb E|\det T(g)|\right)^{1/n}  \geq c' n^{-5/2}.
\end{align*}
Combining all the above we get the desired result. \prend

\smallskip
Following the line of argument of Theorem \ref{thm:adv}, with  
appropriate adjustments provided by the estimates proved in this section, we conclude that Theorem \ref{thm:adv} holds for isotropic, log-concave vectors:

\begin{corollary}
	Let $m,n \in \mathbb N$ and $\delta\in (0,1)$ which satisfy $m\geq C\delta^{-4} n \log(en/\delta)$. Then, for any
	$N \geq m+n$, for any isotropic, log-concave random vector ${\bf a}=(a_j)_{j=0}^N \in \mathbb R^{N+1}$, the 
	numerator $P_{mn}$ of the $[m,n]$-Pad\'e approximant of $f_N(z) = \sum_{j=0}^N a_jz^j$ satisfies
	\begin{align*}
			\mathbb P \left( \frac{\log L(P_{mn})}{m} >\delta^4\right) <\delta.
	\end{align*}
\end{corollary}

\smallskip
\bibliography{prob-edrei-ref}
\bibliographystyle{alpha}



\appendix

\section{Discrepancy of measures} \label{app:disc}

We provide a proof for Proposition \ref{prop:disc}. To this end, we prove the following estimate:

\begin{proposition} [Discrepancy]
	Let $P(z) = \sum_{j=0}^N a_j z^j$ with $a_0a_N\neq 0$, let $\nu_P$ be the normalized zero-counting measure associated with $P$, 
	and let $\mu$ be the uniform probability measure on $\mathbb T=\{z\in \mathbb C : |z|=1\}$. 
	Then, for any function $f : \mathbb C \to \mathbb R$ which is Lipschitz and bounded, we have
	\begin{align}
		\left| \int f \, d\nu_P - \int f\, d\mu \right| \leq 31 \|f\|_{\rm BL}(\log L(P)/N)^{1/4}.
	\end{align}
\end{proposition}

\noindent {\it Proof.} Without loss of generality let $\|f\|_{\rm BL} := \|f\|_\infty +\|f\|_{\rm Lip} =1$ (otherwise we work with $f_1: = f/\|f\|_{\rm BL}$).
To ease the notation let us set $L= \log L(P)$. We assume, as we may, that $L\leq N$, otherwise the result holds trivially.
 Let $0<\rho \leq 1$, $m\geq 1$ (both to be chosen appropriately later), and $Z_P : = \{z \in \mathbb C: P(z) =0 \}$. Note that 
	\begin{align}
		\int f \, d \nu_P = \frac{1}{N} \left [ \sum_{z\in Z_P\cap R(\rho)} f(z) + \sum_{z\in Z_P \setminus R(\rho)} f(z) \right].
	\end{align}
First, Jensen's result (Proposition \ref{prop:J2}) yields 
	\begin{align}
	\left | \sum_{z\in Z_P\setminus  R(\rho)} f(z) \right| \leq \|f\|_\infty |Z_P\setminus R(\rho)| \leq \frac{L}{\rho}.
	\end{align}
If we decompose $R(\rho)$ into $m$ equi-angular polar rectangles, i.e., 
	\[
		R(\rho) = \bigcup_{j=1}^m R(\rho) \cap \left\{ z: \frac{2(j-1)\pi}{m} \leq \arg z < \frac{2 j \pi}{m}  \right\}  \equiv \bigcup_{j=1}^m R_j,
	\] then for each $j=1, \ldots, m$ and for any $w\in Z_P \cap R_j$ we get that
	
	\[
		\left | f(w) - \frac{m}{2\pi} \int_{2\pi(j-1)/m}^{2\pi j/m} f(e^{it}) \, dt \right| \leq \|f\|_{\rm Lip} {\rm diam}(Z_P \cap R_j) \leq 2\pi \left(\frac{1}{m} + \rho \right).
		\footnote{Let $0<a<b$, $0<\theta < \phi <2\pi$, then for $R:= \{re^{it} \mid a<r<b, \, \theta < t < \phi\}$ we have ${\rm diam}(R) \leq |a-b| + (|a| \land |b|)|\theta-\phi|$.}
	\] 
Hence, the choice $\rho = 1/m$ yields for each $j=1, \ldots, m$ that
	\begin{align} \label{eq:A-4}
		\left| \bar{f_j} - \frac{1}{|Z_P\cap R_j|} \sum_{z\in Z_P\cap R_j} f(z) \right| \leq \frac{4\pi}{m}, \quad 
		\bar{f_j}: = \frac{m}{2\pi}\int_{2\pi(j-1)/m}^{2\pi j/m} f(e^{it}) \, dt.
	\end{align}
Now we may write
\begin{equation} \label{eq:A-5}
\begin{split}
	\left | \int f\, d\nu_P - \int f\, d\mu \right| &\leq  \left| \frac{1}{N} \sum_{z\in Z_P\cap R(\rho)} f(z) - \frac{1}{2\pi} \int_0^{2\pi}f(e^{it}) \, dt \right| + \frac{m L}{N} \\
& \leq \sum_{j=1}^m \left| \frac{1}{N} \sum_{z\in Z_P\cap R_j} f(z) -\frac{1}{2\pi} \int_{2\pi(j-1)/m}^{2\pi j/m} f(e^{it}) \, dt \right| +  \frac{m L}{N} \\
& = \sum_{j=1}^m \left| \frac{|Z_P\cap R_j|}{N}  \left( \frac{1}{|Z_P\cap R_j|} \sum_{z\in Z_P\cap R_j} f(z)  - \bar{f_j} \right)+ 
				\left( \frac{|Z_P\cap R_j|}{N}-\frac{1}{m} \right) \bar {f_j} \right| \\
				 &+ \frac{m L}{N} \\
& \stackrel{\eqref{eq:A-4}}\leq \frac{4\pi L}{m} \sum_{j=1}^m \frac{|Z_P \cap R_j|}{N} + m \max_{j\leq m} \left| \nu_P(R_j) - \frac{1}{m}\right| \int |f| \, d\mu + \frac{m L}{N} \\
& \leq \frac{4\pi }{m} + m  \max_{j\leq m} \left| \nu_P(R_j) - \frac{1}{m}\right|  + \frac{m L}{N}.
\end{split}
\end{equation} 
It remains to notice that 
	\[
		\nu_P(R_j) \leq \nu_P \left( z: \frac{2\pi(j-1)}{m} \leq \arg z < \frac{2\pi j}{m} \right) \leq \frac{1}{m} + 16 \sqrt{\frac{L}{N}},
	\] by the Erd\H os-Tur\'an estimate (Proposition \ref{prop:ET}), and
	\begin{align*}
	\nu_P(R_j) & = \nu_P\left( z : \frac{2\pi(j-1)}{m} \leq \arg z < \frac{2\pi j}{m}\right) - \nu_P \left( \left \{z: \frac{2\pi(j-1)}{m} \leq \arg z < \frac{2\pi j}{m} \right\} \setminus R(\rho) \right) \\
			& \geq \frac{1}{m} -16\sqrt{ \frac{L}{N} } - \frac{mL}{N}.   
	\end{align*}
Plugging these estimates into inequality \eqref{eq:A-5} we obtain
	\[
		\left | \int f\, d\nu_P - \int f\, d\mu \right| \leq \frac{4\pi }{m} + 16m\sqrt{\frac{L}{N}} + \frac{2m^2 L}{N},
	\]
and ``optimizing'' over $m$, by choosing $m = (\frac{N}{L})^{1/4}\geq 1$, we conclude the assertion. \prend

\end{document}